\documentclass[final]{dmtcs-episciences}

\usepackage[utf8]{inputenc}
\usepackage{subfigure}

\usepackage[english]{babel}
\usepackage{amssymb, latexsym, amsthm, amsmath, verbatim}
\usepackage{graphicx}
\usepackage{enumerate}
\usepackage[all,cmtip]{xy}
\usepackage{color}
\usepackage{float}
\usepackage{tikz}
\usetikzlibrary{arrows}
\restylefloat{table}
\usepackage{hyperref}

\theoremstyle{theorem}
\newtheorem{thm}{Theorem}[section]

\newtheorem{cor}[thm]{Corollary}
\newtheorem{prop}[thm]{\textbf Proposition}

\theoremstyle{definition}
\newtheorem{defn}[thm]{Definition}

\numberwithin{equation}{section}

\theoremstyle{remark}

\newtheorem{rem}[thm]{\textbf{Remark}}

\newcommand{\abs}[1]{\left| #1\right|}
\newcommand{\pez}[1]{\left(#1\right)}

\usepackage[T1]{fontenc}
\usepackage{babel}

\author{Christopher Coscia\affiliationmark{1}\thanks{Supported by NSF Grant 1263009}
	\and Jonathan DeWitt\affiliationmark{2}\thanks{Supported by NSF Grant 1263009}
	\and Fan Yang\affiliationmark{3}
	\and Yiguang Zhang\affiliationmark{4}\thanks{Supported by Grant No. 14-12 from the Acheson J. Duncan Fund for the Advancement of Research in Statistics at The Johns Hopkins University}}

\title{Best and Worst Case Permutations for Random Online Domination of the Path \thanks{This research began during the Summer 2014 REU at East Tennessee State University.}}

\affiliation{
	Department of Mathematics, Dartmouth College \\
	Department of Mathematics and Statistics, Haverford College \\
	Department of Mathematical Sciences, Carnegie Mellon University \\
	Department of Statistics, Stanford University}


\received{2017-04-21}
\revised{2017-10-6}
\accepted{2017-11-21}

\begin{document}

\publicationdetails{19}{2017}{2}{2}{3278}
\maketitle
\begin{abstract}

We study a randomized algorithm for graph domination, by which, according to a uniformly chosen permutation, vertices are revealed and added to the dominating set if not already dominated. We determine the expected size of the dominating set produced by the algorithm for the path graph $P_n$ and use this to derive the expected size for some related families of graphs. We then provide a much-refined analysis of the worst and best cases of this algorithm on $P_n$ and enumerate the permutations for which the algorithm has the worst-possible performance and best-possible performance. The case of dominating the path graph has connections to previous work of Bouwer and Star, and of Gessel on greedily coloring the path.

\end{abstract}

\section{Introduction}

In this paper, we consider an online algorithm for graph domination, which was introduced in \cite{nikoletseas1994near}. The algorithm is as follows: Let $G$ be a graph on $n$ vertices. Randomly label the vertices with distinct labels $1$ through $n$, and let $v_i$ be the vertex labeled with the number $i$. Let $V_i$ equal $\bigcup_{j=1}^i \{v_j\}$. After $v_i$ is ``revealed,'' we see the entirety of the vertex-induced subgraph $G[V_i]$.  When $G[V_i]$ is revealed, if $v_i$ does not have a neighbor in the dominating set $D$, we add it to $D$. We repeat this procedure until the entire graph has been revealed, and as a result, $D$ is an independent dominating set for $G$. In other words, we reveal the vertices in random order, and, if the revealed vertex is not yet dominated, we add it to the dominating set.

More rigorously, for a graph $G$ on $n$ vertices, the algorithm uses a uniformly random permutation $\pi \in S_n$ to construct a sequence of independent sets $\emptyset = D_0(\pi) \subseteq D_1(\pi) \subseteq D_2(\pi) \subseteq \cdots \subseteq D_n(\pi)$ where $D_n(\pi)$ is an independent \emph{dominating} set for $G$, and returns $D = D_n(\pi)$.  In the $i$th step, put $D_i(\pi) = D_{i-1}(\pi)$ if $v_{\pi_i}$ neighbors any vertex in $D_{i-1}(\pi)$, or $D_i(\pi) = D_{i-1}(\pi) \cup \{v_{\pi_i}\}$ if not.  By construction, every vertex will either be in $D$ or will neighbor a vertex in $D$, but never both, so $D$ is both an independent set and a dominating set for $G$.

For an example of an application of the randomized algorithm described above, consider a street lined with equally-spaced streetlights. The streetlights awaken at night asynchronously and decide whether or not they must illuminate their surroundings. When a streetlight awakens, it senses whether or not its position is already illuminated by another streetlight. If it is not illuminated, the streetlight begins to shine its light; otherwise the light stays off. The expected number of lights on once all of the streetlights awaken is equal to the expectation of the expected size of the dominating set for the path that we compute in this paper, assuming that each streetlight may illuminate those lights to its immediate left and right.  To save energy in a situation such as this, it is preferable that fewer lights are required.

In \cite{nikoletseas1994near}, it is shown that this algorithm is near optimal for dense random graphs, in the sense that $\mathbb{E}|D|$ is close to the domination number of $G$. Using a similar procedure, it is shown in \cite{wieland2001domination} that the domination number of a random graph is concentrated on two numbers. In this paper we study the complementary problem of how well this algorithm performs for specific families of non-dense graphs. We begin by studying the performance of this algorithm on the path on $n$ vertices, $P_n$. We are then able to leverage this information to learn about some other families of graphs. The result of this investigation demonstrates that the algorithm does not perform as well on these sparse graphs as it does on the dense graphs on which it has previously been applied.

Throughout this paper we will use $G$ to denote a graph. 

\begin{defn}
Let $G$ be a graph with vertex set $V$.  We say that $D \subseteq V$ is a \emph{dominating set} for $G$ if its neighborhood is all of $V$: for all $v \in V$, $v \in D$ or $v$ is adjacent to some vertex $w \in V$.
\end{defn}

The following notation will be useful:

\begin{defn}
Let $G$ be a graph with vertex set $\{1, 2, \ldots, n\}$.  Let $\pi \in S_n$ and say that the vertices of $G$ are revealed in the order $\pi_1, \pi_2, \ldots, \pi_n$.  Define $\Gamma(\pi)$ to be the resulting dominating set under the algorithm described above, and let $\gamma(\pi) = |\Gamma(\pi)|$.  We then define the \emph{expected online domination number} of $G$ to be $\gamma_o(G) = \mathbb{E}(\gamma(\pi))$, i.e., the expected size of the dominating set created when our algorithm is run on the graph, with $\pi$ chosen from $S_n$ uniformly at random.
\end{defn}

After determining $\gamma_o(P_n)$ in Section \ref{pathsec}, we then consider in Section \ref{enumsec} the orders, realized by permutations, in which the vertices may be revealed that maximize the size of the dominating set, and enumerate these worst-case permutations. This enumerative work leads to connections between this problem and the work of Bouwer and Star \cite{bouwer1988question} and of Gessel \cite{gessel1991coloring}, which studied cases of greedy colorings of $P_n$ in which only two colors are required; in particular, they enumerate the best-case permutations of path-coloring.  We show that when the length of the path is odd, that the best-case behavior for the path-coloring problem coincides with the worst-case behavior for the domination problem, and when the length of the path is even, that these two problems differ.  We enumerate the number of permutations in the worst case.  We end by enumerating the permutations in which the algorithm gives best-case performance.

The connection between algorithms and permutations has been studied in the past; for instance, it is well-known that stack-sortable permutations are those that avoid the permutation pattern $231$ (\cite{knuth69}). Other connections between sorting procedures and permutation patterns are described in B\'ona's survey \cite{bona2002survey}. In particular, the problem of enumerating the permutations giving best-case and worst-case behavior have also been considered. Moreover, in any satisfying worst-case analysis of an algorithm it is essential to demonstrate that a bound is tight, and so particular instances of worst-case permutations have been studied. It is interesting to know what these permutations look like, especially when they are highly structured. For instance, this has been studied by Elizalde and Winkler in the case of ``homing'' sorting, for which they obtain the upper bound on the worst case of the algorithm and then demonstrate that there are super-exponentially many permutations that obtain this worst case~\cite{elizalde2009sorting}. However, they leave the exact enumeration of this number as an open problem. As an example of a similar problem, we may consider the problem of enumerating extremal Erd\H{o}s-Szekeres permutations, which are those permutations of $\{1,...,n^2\}$ not containing a monotone subsequence of length $n+1$; this enumeration was completed by Romik in \cite{romik2006permutations}.

\section{A Study of $P_n$}\label{pathsec}

In particular, we are interested in computing the asymptotic behavior of $\gamma_o(P_n)$ where $P_n$ is the path on $n$ vertices. For simplicity, we define $\gamma_o(P_n)=0$ for $n\le 0$. For ease of notation, we define $P_n=([n],\{(i,i+1)\mid i\in [n-1]\})$, so that our vertices come pre-labeled, where $[n]$ denotes the set $\{1,...,n\}$. 

\begin{thm}
\[
\lim_{n\rightarrow\infty} \frac{\gamma_o(P_n)}{n}=\frac{e^2-1}{2e^2}\approx .4323...
\]
\end{thm}

\begin{proof}
We first claim that

\begin{equation}\label{eq:rec}
\gamma_o(P_n)=1+\frac{2}{n}\left(\sum_{i=1}^{n-2} \gamma_o(P_i)\right).
\end{equation}

Suppose that the first vertex revealed in the online procedure is vertex $i$. Then $i$ must enter the dominating set. As $i$ is in the dominating set, the vertices $i-1$ and $i+1$ are already dominated, whereas $i-2$ and $i+2$ are not. Thus in order to finish dominating $P_n$ we must separately dominate the two remaining subgraphs, which are isomorphic to $P_{i-2}$ and $P_{n-i-1}$. In particular, if $i=1$ or $n$, one of these paths is empty. Next, note that when we consider each possible $i\in [n]$, there are two instances in which we must dominate $P_j$ for $j\in [n-2]$ -- once if we choose $j+2$ as the first vertex and also if we choose $n-j-1$ as the first vertex. Then as each $i$ is equally likely, the permutation induced on the subgraphs $P_{i-2}$ and $P_{n-i-1}$ is chosen uniformly at random, we sum over these possible $j$, multiply by 2 to account for the two instances for which we dominate each $P_j$, and divide by $n$, giving the formula above.

Now let $F(x)=\sum_{n\ge 0} \gamma_o(P_n)x^n$. We multiply \ref{eq:rec} by $nx^n$ and sum over $n\ge 0$ to find
\[
xF'(x)=\frac{x}{(x-1)^2}+2F(x)\frac{x^2}{1-x}.
\]
Cancelling the $x$, we have the equation
\[
F'(x)=\frac{1}{(x-1)^2}+2F(x)\frac{x}{1-x}
\]
and by solving the differential equation, we have
\[
F(x)=Ce^{-2(x+\log(x-1))}+\frac{1}{2(x-1)^2},
\]
which then allows us to evaluate the function and determine that the constant is $-1/2$. We now wish to determine the coefficients of $F$, which has the form
\[
F(x)=\frac{-1}{2}e^{-2x}\frac{1}{(x-1)^2}+\frac{1}{2(x-1)^2}.
\]
We already know that $n+1$ is the coefficient of $x^n$ in $\frac{1}{(x-1)^2}$, so it will suffice to determine $g_n$, the coefficient of $x^n$ in $e^{-2x}\frac{1}{(x-1)^2}$. To this end, note that
\[
e^{-2x}\frac{1}{(x-1)^2}=\left(\sum_{i=0}^n\frac{(-2)^i}{i!}x^i\right)\left(\sum_{j=0}^n (j+1)x^j\right).
\]
Collecting the $x^n$ terms in the product we have
\[
g_n=\sum_{j=0}^n(n+1-j)\frac{(-2)^j}{j!}
\]
and $\gamma_o(P_n)=\frac{-1}{2}g_n+\frac{n+1}{2}$, thus
\[
\gamma_o(P_n)=\frac{-1}{2}\left(\sum_{j=0}^n(n+1-j)\frac{(-2)^j}{j!}\right)+\frac{n+1}{2}=\frac{-(n+3)}{2}\sum_{j=0}^n\frac{(-2)^j}{j!}+\frac{n+1}{2}.
\]
We are interested in the asymptotic behavior of $\gamma_o(P_n)$, so we consider

\[
\lim_{n\rightarrow\infty} \gamma_o(P_n)/n = \lim_{n\rightarrow\infty} \left(\frac{-(n+3)}{2n}\left(\sum_{j=0}^{n} \frac{(-2)^j}{j!}\right)+\frac{n+1}{2n}\right)
\]
We can then evaluate the limit and recognize the power series as $e^{-2}$ to conclude
\[
\lim_{n\rightarrow\infty} \frac{\gamma_o(P_n)}{n}=\frac{1}{2}-\frac{1}{2e^2}.
\]\end{proof}

Let $C_n$ be the cycle graph on $n$ vertices and $H_n$ the wheel graph with $n$ spokes.  From our analysis of the path graph $P_n$, we can deduce results for these related families of graphs.

\begin{cor}
$\gamma_o(C_n)=1+\gamma_o(P_{n-3})$.
\end{cor}

\begin{proof}
After we add one vertex to the dominating set, the undominated vertices form the graph $P_{n-3}$.
\end{proof}

\begin{cor}
$\gamma_o(H_n)=\frac{1}{n+1}+\frac{n}{n+1}\gamma_o(P_{n-3})$.
\end{cor}

\begin{proof}
If the center is revealed first, every vertex is dominated; otherwise, the center will not be included in the dominating set.  Since the remaining vertices form a path, the result follows.
\end{proof}

Other families of graphs are similarly easy to analyze.  Consider $\star_n$, the star graph with $n$ leaves.  If a leaf is revealed first, then every other leaf must be in the dominating set, and if the center is revealed first, it dominates every other vertex, so 
\[
\gamma_o(\star_n)=\frac{n}{n+1}n+\frac{1}{n+1}=\frac{n^2+1}{n+1}=\Theta(n).
\]

For $K_{\{p_i\}}$ the complete multipartite graph with partitions of size $p_i$, note that once one vertex is placed into the dominating set, all vertices outside that block of the partition are dominated and thus will never be in the dominating set. Thus the remaining vertices in the partition must be added to a dominating set. Thus to find $\gamma_o(K_{\{p_i\}})$, we simply weight by the size of each partition, and so
\[
\gamma_o(K_{\{p_i\}})=\frac{\sum p_i^2}{\sum p_i}.
\]

\begin{rem}
Note that we can obtain a lower bound on these quantities by applying the Caro-Wei bound \cite{caro1979,wei1981}, which applies to a general graph $G$. Let $\pi$ be a permutation of the vertices of $G$ taken uniformly at random, which gives the sequence in which the vertices are revealed. For $v\in V$ define $d^+(v)$ to be the number neighbors of $v$ appearing after $v$ in $\pi$. Clearly we have that $d^+(v)$ is uniformly distributed on $\{0,...,d(v)\}$. Hence by linearity of expectations the expected number of vertices such that $d^+(v)=d(v)$ is $C=\sum_{v\in V}1/(1+d(v))$. Vertices with $d^+(v)=d(v)$ are included in our dominating set, hence $C$ is a lower bound on $\gamma_0(G)$. This gives the bound $\gamma_0(P_n)\ge (n+1)/3$ for the path. 
\end{rem}

\section{Worst Case Permutations}\label{enumsec}

In this section we consider the orders in which vertices may be revealed that maximize the number of vertices included in the dominating set of the graph of $P_n$. When $n$ is even, at most $n/2$ vertices may be included in the dominating set as no two vertices in the dominating set may be adjacent. Further, the permutation $135...246...$ achieves this number.  As any dominating set created by our algorithm is an independent dominating set, and each such dominating set can be achieved (for example, by simply listing those dominating vertices first in the permutation), the number of distinct worst-case dominating sets is equal to the number of maximal independent dominating sets of the graph.

 \begin{prop}\label{prop:maxindsets}
Let $k$ be a positive integer.  If $n = 2k-1$ is odd, there is only one maximal independent dominating set of $P_n$. Moreover, this unique set consists of all odd-numbered vertices.

If $n = 2k$ is even, then the number of maximal independent dominating sets of $P_n$ is equal to $k+1 = n/2+1$.  Moreover, each maximal independent dominating set is either the set of even vertices, the set of odd vertices, or a set of the form $\{1, 3, 5, \ldots, 2j-1, 2j+2, 2j+4, \ldots, n\}$ for some $j \in \{2, 3, \ldots, k-1\}$.
\end{prop}

\begin{proof}
First suppose that the length of the path is $2k-1$. The vertex $1$ must always be in the independent dominating set. If it were not, then $2$ must be in the independent set in order for it to be maximal in size. In this case $3$ cannot be in the set, and so finding the largest independent set on the remaining vertices reduces to considering the path on the vertices $\{4,5,...,2k-1\}$. This new path is of length $2k-2$. Because this is a worst-case permutation, there must be $k$ vertices in the independent set, and so there remain $k-1$ vertices to be added to the independent set. However, at most $k-2$ vertices can be placed on the remaining path.  Thus $1$ must be in the set and inductively every odd vertex in the path must be in the independent set.  We conclude that there is a unique maximal independent set of $P_{2k-1}$.

We now consider the even case, in which the path has $2k$ vertices. Note that if there is no pair of consecutive vertices that are both not in the dominating set, then the dominating set will consist of vertices of a common parity. Otherwise, there is at least one such pair; first suppose there is exactly one, $i+1$ and $i+2$. We claim that $i$ must be odd. If $i$ were even, the remaining undominated entries $1,2,...,i-2$ would form a path of even length, whereas $i+5,i+6,...,2k$ would form a path of odd length. In particular, these paths can accomodate at most $i/2-1$ and $k-i/2-2$ vertices, respectively, in their dominating sets. This sums to $k-3$ vertices. As we have added two vertices to the dominating set already, we will only manage to choose $k-1$ vertices; however, $k$ vertices are required for the set to be maximal, so this is impossible. One may then check that adding every other vertex except leaving an additional gap between $i$ and $i+3$ gives a maximal set, provided that $i$ is odd. We now claim that there can be at most one such pair of consecutive non-dominated vertices. If there were two gaps, then we could take the rightmost vertex of the leftmost gap, and the leftmost vertex of the rightmost gap, and translate those two vertices and all those vertices between them one position to the left. Then there is a vacant spot next to where the rightmost vertex that was moved originated, which we can add to the dominating set, and thus the independent set was not maximal.
\end{proof}

Now that we have determined the possible worst-case configurations, we study which permutations correspond to these configurations. We begin here with the study of the single configuration when $n$ is odd; this is the dominating set $\{1,3,5,...,n\}$. Let $F_n$ be the set of permutations which achieve the worst-case bound. The following table summarizes the values of $\abs{F_n}$ for small values of $n$ and was generated by explicitly testing each permutation in $S_n$. In the case $n=3$ the four permutations in $F_n$ are $123, 132,312,321$, i.e. those that do not start with $2$.

\[
\begin{tabular}{|c|c|}
\hline $n$& $\abs{F_n}$\\ \hline
1 & 1\\
2& 2\\
3&4\\
4& 24\\
5& 56\\
6&640\\
7&1632\\
8& 30464\\
9& 81664\\
10 & 2251008\\
11 & 6241280\\ \hline
\end{tabular}
\]

\begin{prop}\label{prop:oddrecurrence}
Let $f(n) = |F_n|$.  For $n$ odd, $f(n)$ satisfies the recurrence relation

\[ f(n) = 2(n-1)f(n-2) + (n-1)(n-2)\sum_{i=3 \\ \text{, $i$ odd }}^{n-2} \binom{n-3}{i-2}f(i-2)f(n-i-1)\]

or equivalently,

\[ f(2k+1) = 4kf(2k-1) + (4k^2-2k) \sum_{j=1}^{k-1} \binom{2k-2}{2j-1} f(2j-1)f(2k-2j-1). \]

In particular, it is possible to enumerate $|F_n|$ for $n$ odd without knowing any values for $n$ even.
\end{prop}

\begin{proof}
Let $f_i(n) = |\{ \pi \in S_n : \pi \in F_n \text{ and } \pi(1) = i \}|$.   Then $f(n) = \sum_{i=1}^n f_i(n)$ and $f_j(n) = 0$ when $j$ is even.  When $i = 1$ or $i = n = 2k+1$, the first vertex added is an endpoint of the path.  It then remains to dominate $2k-1$ more vertices, and the label corresponding to neighbor of the endpoint of the path may appear anywhere in the permutation (as it will never be placed into the dominating set).  The number of permutations that give rise to a maximal dominating set on the $2k-1$ vertices is $f(2k-1)$, and there are $2k$ possible positions for the neighboring vertex, so $f_1(2k+1) = f_{2k+1}(2k+1) = 2kf(2k-1)$.

For $i$ odd and not 1 or $2k+1$, choosing the $i$th vertex as the first to add splits the path into two parts of odd length.  The first $i-2$ vertices must be dominated, and the final $n-i-1$ vertices must be dominated, while the $(i-1)$st and $(i+1)$st vertices may appear anywhere in the permutation as they will not be in the dominating set.  Thus we must place an ordering on these two sets of vertices and interleave them in any way we wish; there are $\binom{n-3}{i-2}f(i-2)f(n-i-1)$ ways to do this.  We must then choose positions in the permutation for the two neighbors of the first vertex chosen, so in this case $f_i(n) = (n-1)(n-2)\binom{n-3}{i-2}f(i-2)f(n-i-1)$.
\end{proof}

It turns out that the subsequence $\{f(2k+1)\}_{k \in \mathbb{N}}$ has been encountered in conjunction with weakly alternating permutations.  We say that a permutation $\pi \in S_n$ is \emph{weakly alternating} if for every even index $i$ we have either $\pi_{i-1}<\pi_i$ or $\pi_{i+1}<\pi_i$, i.e. there is a weak peak at $i$. Let $W_n$ be the set of weakly alternating permutations of order $n$.

\begin{prop}\label{prop:oddcase}
We have $\abs{F_n}=\abs{W_n}$ for odd $n$. In particular, the map that sends $\pi\mapsto \pi^{-1}$ is a bijection from $F_n$ into $W_n$ for $n$ odd.

\end{prop}
\begin{proof}
Suppose that $\pi\in F_n$ for odd $n$.  The dominating set $\Gamma(\pi)$ consists of the odd-labeled vertices along the path.  For each even $i\in [n]$, one of $i-1$ and $i+1$ must be in the dominating set by the time that $i$ is revealed as otherwise $i$ would be added to the dominating set. Moreover, requiring that one of $i-1$ and $i+1$ appear before $i$ in the permutation is clearly sufficient to ensure that $i$ is not in the dominating set and that every odd vertex is dominated. Now consider $\pi^{-1}$. We claim that $\pi^{-1}$ is weakly alternating. If some even $i\in [n]$ appears at index $k$ in $\pi$ we have either $i-1$ or $i+1$ at some index $j<k$. Without loss of generality suppose that $\pi(j)=i-1$. Then $\pi^{-1}(i-1)=j$ and $\pi^{-1}(i)=k$, and so $\pi^{-1}$ does have a weak peak at every even $i$. Now suppose that $\sigma$ is a weakly alternating permutation. Suppose without loss of generality that $\sigma(i)=k$ and $\sigma(i-1)=j<k$, then $\sigma^{-1}(j)=i-1$ appears before $\sigma^{-1}(k)=i$. Thus the map that sends $\pi\rightarrow \pi^{-1}$ is a bijection between $F_n$ and $W_n$ for odd $n$.
\end{proof}

To enumerate $W_n$ for $n$ odd, recall that the complement $c(\pi) \in S_n$ of a permutation $\pi\in S_n$ is defined by $c(\pi)_i=n+1-\pi_i$ \cite{kitaev2011patterns}. The complement of a weakly alternating permutation with weak peaks at even indices is a permutation with no local maxima in even positions. Permutations of odd length with no local maxima in even positions are enumerated in OEIS sequence \href{https://oeis.org/A113583}{A113583}.

\subsection{Connection with Path Coloring}

Permutations with no local maxima in even positions have been studied previously in relationship to greedy graph colorings of the path. With some predetermined ordering on the vertices, the greedy algorithm for coloring a path labels each vertex, in order, with the lowest number so that no two adjacent vertices receive the same label. In particular, suppose that we color the vertices in the order that they appear in $\pi$. Then the number of permutations in which the odd vertices are colored $1$ and the even vertices are colored $2$ is equal to the number of permutations of odd length with no even local maxima \cite{bouwer1988question}. However, the form of the extremal permutations differs from the graph coloring problem when the length of the path being colored is even. In particular, let $D(n)$ be the number of permutations in $S_n$ which lead to all of the odd vertices of $P_n$ being placed in the dominating set. Gessel discovered in \cite{gessel1991coloring} that if $G(x)$ is the odd part of the exponential generating function ($G(x)=D(1)x/1!+D(3)x^3/3!+\cdots),$

\begin{equation}
G(x)=\frac{\sinh x}{\cosh x-x\sinh x}
\end{equation}
and also that if $H(x)=D(0)/0!+D(2)x^2/2!+D(4)x^4/4!+\cdots$, then
\begin{equation}
H(x)=\frac{1}{\cosh x-x\sinh x}.
\end{equation}
These formulas were studied prior to the work of Gessel in a different form in \cite{bouwer1988question}.

\subsection{Enumerating $F_{2n}$}

We now turn to the remaining problem of enumerating the worst-case permutations of even length. In particular we establish the following.

\begin{prop}
Let $f(n) = |F_n|$ and define $f(0) = 1$.  With values of $f(n)$ for $n$ odd as in Proposition \ref{prop:oddrecurrence}, when $n$ is even, $f(n) = f(2k)$ satisfies the recurrence relation

\[ f(n) = 2(n-1)f(n-2) + (n-1)(n-2) \sum_{i=2}^{n-1} \binom{n-3}{i-2} f(i-2)f(n-i-1) \]

or equivalently,

\[f(2k) = (4k-2)f(2k-2) + (4k^2-6k+2) \sum_{i=2}^{2k-1} \binom{2k-3}{i-2} f(i-2)f(2k-i-1) \]
\end{prop}

The proof is similar to the case that $n$ is odd, except that choosing a vertex other than the first or last splits the path into an odd-length path and an even-length path, which must be dominated separately, and we cannot throw out the cases in which $i$ is even; in fact, by symmetry, the procedure is the same as when $i$ is odd).

\begin{thm}\label{thm:genfunc}
Let $F_{n}$ be the set permutations which achieve the upper bound for the number of vertices in the dominating set when the online algorithm is run on the path. Let $F(x)$ be the exponential generating function for $\abs{F_n}$.  Then
\begin{equation}
F(x)=\frac{\sinh x}{\cosh x-x\sinh x}+\frac{1}{(\cosh x-x\sinh x)^2}.
\end{equation}
\end{thm}

One immediate advantage of this view is that the coefficients of the exponential generating function are exactly the probabilities that a randomly chosen permutation in $S_n$ gives rise to a maximal dominating set by this online algorithm.

\begin{proof}
We have already established that the first term in the sum corresponds to the odd powers of $x$ by our remarks in Proposition \ref{prop:oddcase}, and so the even case is all that remains. To finish the proof we show combinatorially that
\[
\abs{F_{2n}}=\sum_{i=0}^n D(2n-2i)D(2i){2n\choose 2i},
\]
a convolution which gives us the second term in the sum to account for the even powers of $x$.

Suppose that $\pi$ is a permutation in $F_{2n}$. If all of the vertices in $\Gamma(\pi)$ have the same parity, then we have that either $\pi$ is a permutation counted by $D(2n)$ or the reverse of $\pi$ is a permutation counted by $D(2n)$. This deals with the cases in the above equation where $i$ is $0$ or $n$. Otherwise, from Proposition \ref{prop:maxindsets} we know that $\Gamma(\pi)$ has the form $ \{1,3,5,...,2j-1,2j+2,2j+4,...,2n\}$. Now let $\pi'$ be the permutation $\pi$ restricted to the set $[2j]$ and let $\pi''$ be the permutation restricted to the remaining elements. By the form of $\Gamma(\pi)$ we have that $\pi'$ is counted by $D(2j)$. Subtracting $2j$ from the elements of $\pi''$ and reversing their order provides similar inclusion in $D(2n-2j)$. Moreover, consider taking two permutations $\sigma$ counted by $D(2n-2j)$ and $\tau$ counted by $D(2j)$ for $0<j<n$. Reverse the entries of $\tau$, add $2n-2j$ to each entry and now place the entries of this new permutation among the entries of $\sigma$ to create a permutation $\pi$. This map is clearly an injection and moreover there are ${2n\choose 2j}$ ways of merging the two permutations as each way of merging them is selected by the choice of $2i$ entries which the image of $\tau$ occupies.  For example, if we take $\sigma=534216$ and $\tau=135642$ then $\tau\mapsto 246531\mapsto 8(10)(12)(11)97$. We can then join these permutations, for instance, as
\[
\pi=58(10)342(12)(11)1697\text{ and } \Gamma(\pi)=\{1,3,5,8,10,12\}.
\]
The vertices $2j$ and $2j+1$ are never placed into the dominating set by this procedure as they never appear before $2j-1$ or $2j+2$, respectively, in $\pi$. Then as neither of $2j$ and $2j+1$ are in the dominating sets, the vertices less than $2j$ are effectively dominated independently from the vertices greater than $2j+1$.
\end{proof}

The proofs of the formulas for $G(x)$ and $H(x)$ given in \cite{gessel1991coloring} are purely combinatorial, and so as the above proof is also combinatorial, Theorem \ref{thm:genfunc} has been proven by purely combinatorial means.

\section{Best Case Permutations}

In this section we enumerate the permutations for which the algorithm gives the optimal result. Given a permutation $\pi\in S_n$ with corresponding dominating set $\Gamma(\pi)$, each vertex of $\Gamma(\pi)$ dominates at most three elements of $P_n$, and thus the minimum size of a dominating set is $\lceil n/3\rceil$. Dominating sets achieving this bound are easy to construct: pick all of the vertices with a label congruent to $2$ modulo $3$ as well as the final vertex. We enumerate these permutations in three different cases. In particular, the resulting formula depends strongly on the congruence class of $n$ modulo $3$. Let $B_n$ be the set of permutations of length $n$ which achieve the lower bound on the size of the dominating set.  

Throughout this section, we will lean heavily on the following idea:  when we run the online algorithm on a graph $G$ and hope to produce a best-case independent dominating set, if at some point have an independent (not yet dominating) set $D'$, in order to minimally dominate $G$ moving forward we must minimally dominate all of the components of $G \smallsetminus H$, where $H$ is the graph induced by the neighborhood of $D'$.  In other words, a partial independent dominating set breaks the graph into components each of which can be considered independently of the others.

\begin{prop}
For $n\equiv 0 \bmod 3$, we have that
\begin{equation}
\abs{B_n}=\frac{n!}{3^{n/3}}.
\end{equation}
\end{prop}
\begin{proof}
Note that for any dominating set $\Gamma(\pi)$ for $\pi\in B_n$ we have that $\Gamma(\pi)=\{2,5,8,...,n-1\}$ as each vertex is dominated exactly once. 
Choose a permutation $\pi \in S_n$ uniformly at random.  The probability that $\pi \in B_n$ is exactly the probability that $3k+2$ appears before both $3k+1$ and $3k+3$ for each $k \in \{0, \ldots, \frac{n}{3}-1\}$; this is exactly $\frac{1}{3^{n/3}}$, and the result follows by linearity of expectation.

\end{proof}

We next consider permutations of $n$ for which $n\equiv 2 \mod 3$.
\begin{prop}\label{prop:nequiv2mod3}
For $n\equiv 2\bmod 3$, we have that for $n>2$,
\[
\abs{B_n}=24{n\choose 5} \frac{(n-5)!}{3^{(n-5)/3}}\pez{\frac{n-2}{3}}+2{n\choose 2}\frac{(n-2)!}{3^{(n-2)/3}}.
\]
\end{prop}

\begin{proof}
We consider two cases separately; either $\Gamma(\pi)$ contains one of $1$ or $n$ or it does not. In the case that $\Gamma(\pi)$ contains $1$ then $\pi$ restricted to $\{3,...,n\}$ is dominated minimally. In particular, this set contains a number of vertices congruent to $0$ modulo $3$ so the permutation $\pi$ when restricted to $\{3,...,n\}$ is of the type in $B_{n-2}$. We must place the entries $1$ and $2$ in order within such a permutation, thus there are
\[
{n\choose 2}\frac{(n-2)!}{3^{(n-2)/3}}
\]
such permutations. The analysis when $n\in \Gamma(\pi)$ is similar. 

Next suppose that neither $1$ nor $n$ are in the dominating set. We claim that there exists a unique vertex which is dominated by two vertices. Each vertex dominates three vertices, and as there are $\lceil n/3\rceil $ members of the dominating set and this number is greater than $n/3$ by exactly $1/3$, there must be a unique vertex that is dominated by both of its neighbors. We now count permutations $\pi\in B_n$ which have a unique pair of indices $i,i+2$ in their dominating set. For such a permutation, note that if we restrict $\pi$ to $\{1,...,i-2\}$ and $\{i+4,...,n\}$, each of these sets must contain a number of vertices congruent to $0$ modulo $3$ as otherwise the remaining vertices in the dominating set cannot possible dominate the rest of the graph. We can count the number of ways to dominate the vertices $\{1,...,i-2\}$ and $\{i+4,...,n\}$ as in the previous proposition.  There are $24$ permutations $\sigma$ of $\{i-1,i,i+1,i+2,i+3\}$ such that $\Gamma(\sigma) = \{i, i+2\}$, and there are ${n\choose 5}$ ways of combining $\sigma$ and the permutation defined on the remaining vertices. There are $(n-2)/3$ ways to choose $i$. Multiplying all of these choices together then gives the result.
\end{proof}

Finally we consider the hardest case, $n\equiv 1\mod 3$. In the previous two cases, the permutations were highly structured, as there was at most one vertex dominated by both of its neighbors.  This case is more complicated because this restriction is slightly loosened.

\begin{prop}
For $n\equiv 1 \bmod 3$, we have that for $n>7$,
\begin{align*}
\abs{B_n}&=720{n\choose 7}\frac{n-4}{3}\frac{(n-7)!}{3^{(n-7)/3}}+24^2{10\choose 5}{n\choose 10}{(n-4)/3\choose 2}\frac{(n-10)!}{3^{(n-10)/3}}+6{n\choose 4}\frac{(n-4)!}{3^{(n-4)/3}}\\
&+2\pez{9{n\choose 4}\frac{(n-4)!}{3^{(n-4)/3}}+{24{n\choose 2} {n-2\choose 5} \frac{(n-7)!}{3^{(n-7)/3}}\frac{n-4}{3}}}.
\end{align*}
\end{prop}

\begin{proof}
We enumerate different subsets of $B_n$ separately, with the subsets essentially defined in terms of how each member dominates the end vertices. We begin by considering those $\pi\in B_n$ such that neither $1$ nor $n$ is contained in $\Gamma(\pi)$. Let $k$ be such that $n = 3k+1$.  As each vertex of the dominating set dominates three vertices (including itself) and there are $k+1$ vertices in the dominating set, there exist exactly two vertices $i, j \notin \Gamma(\pi)$ that are dominated by two other vertices. We further separate into two cases, depending on whether or not $|i-j| = 2$.

Suppose first that (without loss of generality) $j-i = 2$, then the dominating set contains the three vertices $i-1,i+1,i+3$. Moreover, note that if we restrict $\pi$ to $\{1,...,i-3\}$ and $\{i+5,...,n\}$ that there are $n-7$ vertices left to dominate with $k - 2$ vertices remaining to put into the dominating set. In particular, this shows that we must have a multiple of three vertices in each of these blocks, and also, because we have so few remaining vertices, the dominating set is completely determined. Moreover, computing by brute force gives that there are $720$ ways to dominate the vertices $i-1,...,i+5$ with $i, i+2,$ and $i+4$. Then there are $k-1 = (n-4)/3$ possible values that $i$ may take, so we have the first term in the formula.

Next suppose that $|j-i| \neq 2$. Then there are exactly two pairs $i-1,i+1$ and $j-1,j+1$ in $\Gamma(\pi)$, and they are disjoint. There are $24$ ways to minimally dominate each block $\{i-2,...,i+2\}$ (necessarily including $i-1$ and $i+1$ in the dominating set) individually. Moreover, by counting the number of vertices that need to be dominated outside of the blocks we know that each of the sets $\{1,...,i-3\}, \{i+3,...,j-3\}$ and $\{j+3,...,n\}$ has size a multiple of three. There are $\binom{k-1}{2}$ ways to pick $i$ and $j$. Then multiplying all of this information together as before gives the second term in the expression.

Now suppose that both $1$ and $n$ are contained within the dominating set. Then $\pi$ restricted to $\{3,...,n-3\}$ must dominate $n-4$ vertices with $k-1$ vertices, so we may treat these remaining vertices as though they belong to the case where $n\equiv 0\bmod 3$. Note that there are $6$ ways to dominate $\{1,2,n-1,n\}$ with 1 and $n$ so joining these with a permutation that dominates $\{3,...,n-3\}$ gives the second to third term in the equation.

Next suppose that $1$ is in the dominating set but that $n$ is not (the case in which $n$ is in the dominating set but $1$ is not is analogous, so we multiply the relevant terms by $2$). Consider the subcase in which 3 is also in the dominating set. Restricting to the vertices $\{5,...,n\}$, we have $n-4$ vertices to dominate and $k - 1=(n-4)/3$ vertices to add to the dominating set. Thus we may consider $\pi$ restricted to $\{5,...,n\}$ as being a permutation in $B_{n-4}$. There are $9$ ways to put $1$ and $3$ but not $2$ and $4$ in the dominating set. Thus combining as before we obtain the first summand in the final term.

Finally, suppose that $1$ is in the dominating set but 3 and $n$ are not. In this case we consider the restriction of $\pi$ to $\{3,...,n\}$ and identify this with a permutation in $B_{n-2}$. Each such permutation arises by combining a permutation in $B_{n-2}$ which does not include $1$ or $n-2$ in its dominating set with the permutation $12 \in S_2$. From our analysis in Proposition \ref{prop:nequiv2mod3}, we know such permutations have a pair of entries in their dominating set $i,i+2$. Selecting that pair of entries and enumerating as before gives us the second summand within the parentheses. As these different cases are exhaustive, we have finished the proof.

\end{proof}

For completeness, we also include the following table, which gives the values of $\abs{B_n}$ for small $n$.


\[
\begin{tabular}{|c|c|}
\hline $n$& $\abs{B_n}$\\ \hline
1 & 1\\
2& 2\\
3&2\\
4& 24\\
5& 64\\
6&80\\
7&3408\\
8& 9856\\
9& 13440\\
10 & 1377792\\
11 & 4139520\\ \hline
\end{tabular}
\]

\section{Further Questions}
There are a number of questions which arise from study of this random online algorithm. One obvious question concerns the enumeration of other types of permutations with respect to the number of vertices they will cause to lie in the dominating set. 

\begin{figure}[h!]
  \caption{$\gamma(\pi)$ plotted for $40000$ randomly chosen permutations $\pi$ in $S_{2000}$. }
  \centering
    \label{fig5}\includegraphics[width=0.5\textwidth]{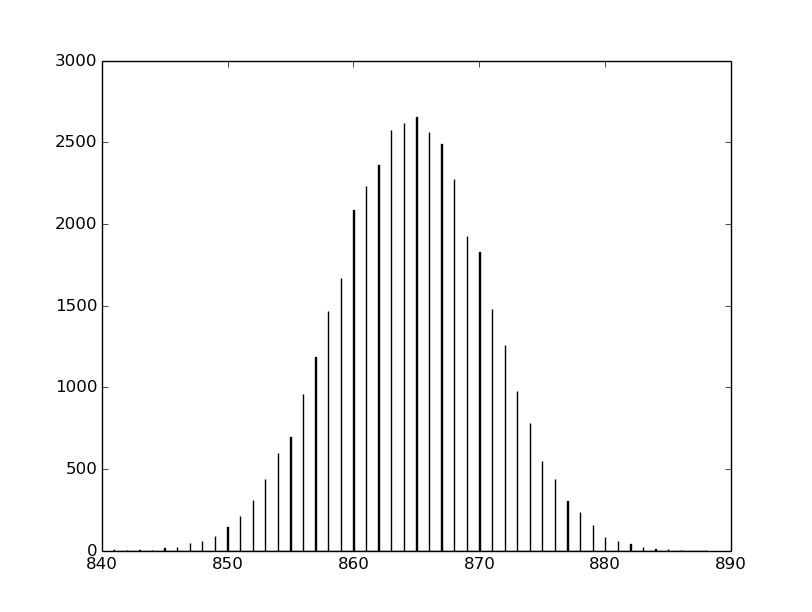}
\end{figure}

There is another question which is more closely related to the probabilistic section of the paper. Given the distribution of $S_n$ under the map $\gamma$ on the line, one can appropriately normalize and then ask whether this distribution is converging to some limit distribution. We have done some tests of this and the results seem encouraging, cf. Figure \ref{fig5}. Note that because $\gamma(\pi)$ for $\pi\in S_n$ is between about $n/2$ and $n/3$ that normalizing by $n$ and translating will result in a limiting distribution with compact support, which will not be normal.

\section{Acknowledgments}
We thank Anant Godbole for his suggestion of the problem. We also thank Curtis Greene for suggesting that we consider best and worst case permutations.

\bibliographystyle{plain}

\bibliography{online}

\end{document}